\documentclass[a4paper,USenglish]{lipics}
\pdfoutput=1
\usepackage{microtype}
\usepackage[T1]{fontenc}
\usepackage[utf8]{inputenc}
\usepackage[all]{xy}
\usepackage{bussproofs}
\usepackage{array}
\usepackage{cmll}
\usepackage{mathtools}
\usepackage{hyperref}
\mathtoolsset{
showonlyrefs=true 
}
\EnableBpAbbreviations

\newcommand{\tend}{\mathsf{end}}
\newcommand{\readb}{\mathsf{r}}
\newcommand{\writez}{\mathsf{w\!0}}
\newcommand{\writeo}{\mathsf{w\!1}}

\newcommand{\zos}{\{0,1\}^*}

\newcommand{\kkproc}{\mathsf{\kklam\!\star\!\kkpi}}
\newcommand{\kkproce}{\mathsf{P}}
\newcommand{\kkpole}{\kkbot}

\newcommand{\kklame}{\kklam_e}

\newcommand{\kkpie}{\kkpi_e}
\newcommand{\kklamp}{\kklam_p}
\newcommand{\kkpiz}{\kkpi_0}

\newcommand{\bequiv}{\simeq_\beta}
\newcommand{\gequiv}{\simeq_\gamma}
\newcommand{\tequiv}{\sim_\top}
\newcommand{\bin}{\mathrm{bin}}
\newcommand{\lwz}{\mathsf{w0}}
\newcommand{\lwo}{\mathsf{w1}}
\newcommand{\lrz}{\mathsf{r0}}
\newcommand{\lro}{\mathsf{r1}}
\newcommand{\lre}{\mathsf{r\ve}}
\newcommand{\lend}{\mathsf{e}}
\newcommand{\actions}{\mathit{Act}}
\newcommand{\labels}{\mathcal{L}}
\newcommand{\cn}[1]{\overline{#1}}

\newcommand{\kred}{\succ}
\newcommand{\kexec}{\leadsto}
\newcommand{\enti}{\vdash_I}
\newcommand{\entj}{\vdash_J}
\newcommand{\wiggles}{\leadsto^*}
\newcommand{\chole}{[\cdot]}
\newcommand{\porth}{^\Bot}
\newcommand{\toppi}{\to P(\kkpi)}
\newcommand{\ttrans}{\xrightarrow{\tau}}
\newcommand{\ttranss}{\xRightarrow{\tau}}
\newcommand{\atrans}{\xrightarrow{\alpha}}
\newcommand{\atranss}{\xRightarrow{\alpha}}
\newcommand{\R}{\mathbb{R}}
\newcommand{\omicron}{o}
\newcommand{\vertor}{\mathrel|}
\newcommand{\comccc}{\mathsf{c\!c}}
\newcommand{\qdot}{\,.\,}
\newcommand{\comk}{\mathsf{k}}				% combinator k
\newcommand{\kkpi}{\mathsf{\Pi}}
\newcommand{\ap}{\mathclose\cdot}			% application in pca's
\newcommand{\kklam}{\mathsf{\Lambda}}
\newcommand{\ve}{\varepsilon}
\newcommand{\N}{\mathbb{N}}
\newcommand{\pto}{\rightharpoonup}
\newcommand{\dom}{\mathrm{dom}}
\newcommand{\subs}{\subseteq}
\newcommand{\qimp}{\quad\imp\quad}
\newcommand{\imp}{\Rightarrow}
\newcommand{\qtext}[1]{\quad\text{#1}\quad}
\newcommand{\wiggle}{\rightsquigarrow}
\newcommand{\setof}[2]{\{#1\msep #2\}}
\newcommand{\real}{\Vdash}
\newcommand{\msep}{\mathrel|\,}		% separation symbol in subset
\newcommand{\vtp}{\,\mathclose:\,}				% colon for typing of
\newcommand{\kkpl}{\mathsf{PL}}
\newcommand{\ent}{\vdash}
\newcommand{\fa}{\;\forall}
\newcommand{\csep}{\mathrel|\,}		% separation of
\newcommand{\defcon}{\quad\equiv\quad}
\newcommand{\trip}{{\EuScript{P}}}
\newcommand{\setord}{\catset\op\to\catord}
\newcommand{\catset}{{\mathbf{Set}}}
\newcommand{\catord}{\mathbf{Ord}}
\newcommand{\op}{^\mathsf{op}}				% for opposite category
\newcommand{\adj}{\dashv}
\newcommand{\prop}{{\mathsf{Prop}}}		% object of propositions in a tripos
\newcommand{\triptr}{{\mathsf{tr}}}		% generic predicate in a tripos
\newcommand{\id}{\mathrm{id}}				% identity map
\newcommand{\ex}{\;\exists}

\newcommand{\push}{\ap}
\newcommand{\kkbot}{\Bot}

\title{Realizability Toposes from Specifications
\footnote{
This work is supported by the Danish Council for Independent Research Sapere Aude grant ``Complexity
through Logic and Algebra'' (COLA).
}
}

\author[1]{Jonas Frey}
\affil[1]{Department of Computer Science\\
University of Copenhagen, Denmark\\
  \texttt{jofr@di.ku.dk}}
\authorrunning{J. Frey} 
\Copyright{Jonas Frey}
\begin{document}

\maketitle

\begin{abstract}
We investigate a framework of Krivine realizability with
I/O effects, and present a method of associating 
realizability models to \emph{specifications} on the I/O 
behavior of processes, by using adequate interpretations of 
the central concepts of \emph{pole} and \emph{proof-like 
term}. This method does in particular allow to associate 
realizability models to computable functions. 

Following 
recent work of Streicher and others we show how these models 
give rise to \emph{triposes} and \emph{toposes}.
\end{abstract}

\section{Introduction}

Krivine realizability with side effects has been introduced by Miquel 
in~\cite{miquel2009modaleffects}.
In this article we demonstrate how an instance of Miquel's framework including
I/O instructions  allows to associate \emph{realizability toposes} to
\emph{specifications}, i.e.\ sets of requirements imposed on the I/O behavior of
programs. Since the requirement to compute a specific function $f$ can be viewed
as a specification, we do in particular obtain a way to \emph{associate toposes
to computable functions}. 

These toposes are different from traditional 
`Kleene' realizability toposes such as the \emph{effective 
topos}~\cite{hyland82} in that we 
associate toposes to \emph{individual} computable functions, whereas the 
effective topos incorporates \emph{all} recursive functions on equal footing. 
Another difference 
to the toposes based on Kleene realizability is that the internal logic of the
latter is \emph{constructive}, whereas the present approach is based on
Krivine's realizability interpretation~\cite{krivine2009realizabilitypanorama},
which validates classical logic.

To represent specifications we make use of the fact that Krivine's 
realizability interpretation is parametric over a set of processes called the 
\emph{pole}. The central observation (Lemma~\ref{lem:spec-pole} and
Theorem~\ref{theo:f-com-con}) is that non-trivial specifications on program 
behavior give rise to poles leading to consistent (i.e.\ non-degenerate)
interpretations.

To give a categorical account of Krivine realizability we follow recent work of 
Streicher~\cite{streicher2013krivine} and 
others~\cite{stekelenburg2013realizability,vanoosten2012classical,ferrer2014ordered}, 
which demonstrates how Krivine realizability models give rise to \emph{triposes}.
Toposes are then obtained via the tripos-to-topos construction~\cite{hjp80}.

Our basic formalism is an extension of the Krivine 
machine~\eqref{eq:red-rel} that gives an operational semantics to I/O
instructions for single bits. We give two formulations of the operational
semantics -- one \eqref{eq:ex-rel} in terms of a transition relation on 
processes including a \emph{state} (which is adequate for reasoning about 
function computation), and one~\eqref{eq:transsys} in terms of a labeled 
transition system admitting to reason about program equivalence in terms
of bisimulation. The two operational semantics are related by
Corollary~\ref{cor:tequiv}, which we use to prove a Turing completeness
result in Theorem~\ref{theo:turing}.

\subsection{Related work}

The idea of adding instructions with new evaluation rules to the machine
plays a central role in Krivine's writings, as a means to realize non-logical
axioms. Citing from~\cite{krivine2010ralgs}:
\begin{quotation}
\emph{``Indeed, when we realize usual axioms of mathematics, we need to 
introduce, one after the other, the very standard tools in system programming: 
for the law of Peirce, these are continuations (particularly useful for
exceptions); for the axiom of dependent choice, these are the clock and the
process numbering; for the ultrafilter axiom and the well ordering of $\R$, 
these are no less than I/O instructions on a global memory, in other
words assignment.''}
\end{quotation}
Although features like exceptions and memory are often called~\emph{effects},
it is arguable whether they should be called~\emph{side} effects, since they
do not interact with the outside world.

The idea to add instructions for \emph{side} effects which are influenced by -- 
and influence -- the outside world,p has already been investigated by 
Miquel~\cite[Section~2.2]{miquel2009modaleffects}, and our execution
relation~\eqref{eq:ex-rel} can be viewed as an instance of his framework.

What sets the present approach apart is that Miquel views the state of the world 
(represented by a forcing condition) as being \emph{part} of a process and 
requires poles to be saturated w.r.t.\ all (including effectful) reductions, 
whereas for us poles are sets of `bare' processes without state, which are 
saturated only w.r.t.\ reduction free of side-effects.

This difference is crucial in that it enables the construction of poles from
specifications.

\section{Syntax and machine}\label{sec:synmach}

In this section we recall Krivine's abstract machine with continuations 
as described in~\cite{krivine2009realizabilitypanorama}. We then go on to 
describe an extension of the syntax by I/O instructions, and
describe an operational semantics as a transition relation on triples
$(p,\iota,\omicron)$ of process, input, and output.

\subsection{Krivine's machine}\label{sec:kam}

We recall the underlying syntax and machine of Krivine's classical realizability
from~\cite{krivine2009realizabilitypanorama}. The syntax consists of three 
syntactic classes called \emph{terms, stacks}, and \emph{processes}.
\begin{equation}\label{eq:grammar-basic}
\begin{tabular}{l>{$}l<{$}@{$\;::=\;$}>{$}l<{$}l}
Terms: & t &  x\vertor \lambda x\!\qdot\! t\vertor t{}\,t\vertor\comccc \vertor\comk_\pi\\
Stacks: & \pi & \pi_0\vertor t\push\pi & {$t$ closed, $\pi_0\in\kkpiz$} \\
Processes: & p & t\star\pi& {$t$ closed}
\end{tabular}
\vspace{-10mm}
\end{equation}
Thus, the terms are the terms of the $\lambda$-calculus, augmented by a constant
$\comccc$ for call/cc, and continuation terms $\comk_\pi$ for any stack $\pi$.
A stack, in turn, is a list of closed terms terminated by an element $\pi_0$ of a 
designated set $\kkpiz$ of \emph{stack constants}. A process is a pair $t\star\pi$
of a closed term and a stack. The set of closed terms is denoted by $\kklam$, 
the set of stacks is $\kkpi$, and the set of processes is $\kkproc$. 

Krivine's machine is now defined by a transition relation $\kred$ on processes
called \emph{evaluation}.
\begin{equation}\label{eq:red-rel}
{\begin{tabular}
{c>{$}r<{$}@{$\;\star\;$}>{$}l<{$}@{$\quad\kred\quad$}>{$}r<{$}@{$\;\star\;$}>{$}l<{$}}
(push) & tu&\pi& t& u\ap\pi\\
(pop)  & (\lambda x\qdot t[x])& u\ap\pi& t[u]&\pi\\
(save) & \comccc& t\ap\pi& t&\comk_\pi\ap\pi\\
(restore) & \comk_\pi& t\ap\rho& t&\pi
\end{tabular}}
\end{equation}
The first two rules implement \emph{weak head reduction} of $\lambda$-terms,
and the third and fourth rule capture and restore continuations.

\subsection{The machine with I/O}

To incorporate I/O we modify the syntax as follows:
\begin{equation}\label{eq:grammar-extended}
\begin{array}{ll@{\;::=\;}l@{\qquad}l}
\text{Terms:} & t &  x\vertor \lambda x\!\qdot\! t\vertor t{}\,t\vertor\comccc 
\vertor\comk_\pi\vertor\readb\vertor\writeo\vertor\writez\vertor\tend
\\
\text{Stacks:} & \pi & \ve\vertor t\push\pi& t\text{ closed} \\
\text{Processes:} & p & t\star\pi\vertor\top& t\text{ closed}
\end{array}
\end{equation}
The grammar for terms is extended by constants $\readb,\writez,\writeo,\tend$
for reading, writing and termination, and in exchange the stack constants are
omitted -- $\ve$ is the empty stack. Finally there is a \emph{process constant}
$\top$ also representing termination -- the presence of both $\tend$ and $\top$
will be important in Section~\ref{sec:bisim}.

We write $\kklame$ and $\kkpie$ for the sets of terms and stacks of the
syntax with I/O, and $\kkproce$ for the set of processes. Furthermore, we denote
by $\kklamp$ the set of \emph{pure} terms, i.e.\ terms not containing any of
$\readb,\writez,\writeo,\tend$.

The operational semantics of the extended syntax is given in terms of 
\emph{execution contexts}, which are triples $(p,\iota,\omicron)$ of a process 
$p$, and a pair $\iota,\omicron\in\zos$ of binary strings representing input and 
output. On these execution contexts, we define the \emph{execution relation} 
$\kexec$ as follows:
\begin{equation}\label{eq:ex-rel}
\begin{array}
{c@{\qquad(}r@{\;\star\;}r@{,\;}r@{,\;}r@{)\quad\kexec\quad(}r@{\;\star\;}r@{,\;}r@{,\;}r@{)\qquad}l} 
(\tau)  & t & \pi & \iota & \omicron &u & \rho & \iota & \omicron & \text{whenever } t\star\pi\kred u\star\rho \\
(\lrz)  & \readb & t\ap u\ap v\ap\pi&0\iota&\omicron   & t&\pi&\iota&\omicron   \\
(\lro)  & \readb& t\ap u\ap v\ap\pi&1\iota&\omicron & u&\pi&\iota&\omicron   \\
(\lre)  & \readb& t\ap u\ap v\ap\pi&\ve&\omicron & v&\pi&\ve&\omicron   \\
(\lwz)  & \writez& t\ap \pi&\iota&\omicron & t&\pi&\iota&0\omicron   \\
(\lwo)  & \writeo& t\ap \pi&\iota&\omicron & t&\pi&\iota&1\omicron   \\
(\lend) & \tend&\pi&\iota&\omicron & \multicolumn{2}{l@{,\;}}{\top}&\iota&\omicron   \\
\end{array}
\end{equation}
Thus, if there is neither of $\readb,\writez,\writeo,\tend$ in head position,
the process is reduced as in~\eqref{eq:red-rel} without changing $\iota$ and
$\omicron$.
If $\readb$ is in head position, the computation selects one of the first three
arguments depending on whether the input starts with a 
$0$, a $1$, or is empty. $\writez$ and $\writeo$ write out $0$ and $1$, and $\tend$
discards the stack and yields $\top$, which represents successful termination.

We observe that the execution relation is \emph{deterministic}, i.e.\
for every execution context there is at most one transition possible, which
is determined by the term in head position, and in case of $\readb$ also by
the input.

\subsection{Representing functions}

We view the above formalism as a model of computation that explicitly
includes reading of input, and writing of output. 

Consequently, when thinking about expressivity we are not so much interested
in the ability of the machine to transform abstract representations of data like
`Church numerals', but rather in the functions on binary strings that 
processes can compute by reading their argument from the input, and writing
the result to the output.
\begin{definition}\label{def:compute}$ $
For $n\in\N$, $\bin(n)\in\zos$ is the base $2$ representation of $n$. 
$0$ is represented by the empty string, thus we have e.g.\
$\bin(0)=\ve$, $\bin(1)=1$, $\bin(2)=10$, $\bin(3)=11$, \dots

A process $p$ is said to \emph{implement} a partial function 
$f:\N\pto\N$, if 
$ (p,\bin(n),\ve)\wiggles(\top,\ve,\bin(f(n)))$
for all $n\in\dom(f)$.
\end{definition}
\begin{remark}
There is a stronger version of the previous definition which requires 
$(p,\bin(n),\ve)$ to diverge or block for $n\not\in\dom(f)$, and
a completeness result like Thm.~\ref{theo:turing} can be shown for the
strengthened definition as well.

We use the weaker version, since we expect the poles $\kkpole_f$ defined in 
Section~\ref{se:polespec} to be better 
behaved this way.
\end{remark}

\subsection{\texorpdfstring{$\beta$}{β}-reduction}\label{sec:beta}

To talk about contraction of single $\beta$-redexes which are not necessarily in 
head position in a process $p$, we define \emph{contexts} -- which are 
terms/stacks/processes with a single designated hole $\chole$ in term position 
-- by the following grammar:
\begin{equation}
\begin{tabular}{l>{$}r<{$}@{$\quad::=\quad$}>{$}l<{$}l}
Term contexts: & t\chole &  \chole\vertor \lambda x\!\qdot\! t\chole\vertor 
t\chole{}\,t\vertor t\,t\chole
\vertor\comk_{\pi\chole}
\\
Stack contexts: & \pi\chole & t\push\pi\chole\vertor t\chole\push\pi&{$t$,$t\chole$ closed} \\
Process contexts: & p\chole & t\chole\star\pi\vertor t\star\pi\chole
\end{tabular}
\end{equation}
Contexts are used to talk about substitution that allows capturing of variables
-- as described in~\cite[2.1.18]{barendregt1984lambda}, given a context 
$t\chole$/$\pi\chole$/$p\chole$ and a term $u$, $t[u]$/$\pi[u]$/$p[u]$ is the result of 
replacing the hole $\chole$ in $t\chole$/$\pi\chole$/$p\chole$ by $u$, allowing 
potential free variables in $u$ to be captured. We say that
$u$ is \emph{admissible} for $t\chole$/$\pi\chole$/$p\chole$, if 
$t[u]$/$\pi[u]$/$p[u]$ is a valid term/stack/process conforming
to the closedness condition for terms making up stacks.

Now we can express $\beta$-reduction as the action of contracting a single
redex: given a redex $(\lambda x\qdot u)v$ which is admissible for a context 
$t\chole$/$\pi\chole$/$p\chole$, we have
\begin{equation}
 t[(\lambda x\qdot u)v]
\to_\beta
t[u[v/x]]\qquad
 \pi[(\lambda x\qdot u)v]
\to_\beta
\pi[u[v/x]]\qquad
 p[(\lambda x\qdot u)v]
\to_\beta
p[u[v/x]],
\end{equation}
and any single $\beta$-reduction can uniquely be written this way.
$\beta$-equivalence $\bequiv$ is the equivalence relation generated by
$\beta$-reduction.

\section{Bisimulation and \texorpdfstring{$\top$}{T}-equivalence}\label{sec:bisim}

To reason efficiently about execution of processes with side effects --
in particular to show Turing completeness in Section~\ref{sec:exp} --
we want to show that although the computation model imposes a deterministic
reduction strategy, we can perform $\beta$-reduction anywhere in a
process without changing its I/O behavior.

The natural choice of concept to capture `equivalence of I/O behavior' is
\emph{weak bisimilarity} (see~\cite[Section~4.2]{milner1990operational}), and in 
order to make this applicable to processes we have to reformulate the 
operational semantics as a \emph{labeled transition system} (LTS).

We use the set $\labels=\{\lrz,\lro,\lre,\lwz,\lwo,\lend\}$
of labels, where $\lrz$, $\lro$ represent reading of a $0$ or $1$,
respectively, and $\lwz$, $\lwo$ represent writing of bits. $\lre$ represents
the unsuccessful attempt of reading on empty input, and $\lend$ represents
successful termination. The set $\actions=\labels\cup\{\tau\}$ of 
\emph{actions} contains the labels as well as the symbol $\tau$ representing
a `silent' transition, that is used to represent effect-free evaluation.

The transition system on processes is now given as follows.
\begin{equation}\label{eq:transsys}
\begin{array}
{r@{\,\star\,}l @{\,\,}c@{\,\,} r@{\,\star\,}l@{\qquad}r@{\,\star\,}l @{\,\,}c@{\,\,} r@{\,\star\,}l@{\qquad}r@{\,\star\,}l @{\,\,}c@{\,\,} r@{\,\star\,}l}
  (\lambda x\qdot t[x])& t\ap\pi&\xrightarrow{\tau}& t[u]&\pi  & \readb& t\ap u\ap v\ap\pi&\xrightarrow{\lrz}&t&\pi & \writez& t\ap\pi&\xrightarrow{\lwz}&t&\pi \\
  tu&\pi&\xrightarrow{\tau}& t& u\ap\pi                        & \readb& t\ap u\ap v\ap\pi&\xrightarrow{\lro}&u&\pi & \writeo& t\ap\pi&\xrightarrow{\lwo}&t&\pi \\
  \comccc& t\ap\pi&\xrightarrow{\tau}& t&\comk_\pi\ap\pi       & \readb& t\ap u\ap v\ap\pi&\xrightarrow{\lre}&v&\pi & \tend&\pi&\xrightarrow{\lend}&\multicolumn{2}{l@{\;}}{\top} \\
  \comk_\pi& t\ap\rho&\xrightarrow{\tau}& t&\pi
\end{array}
\end{equation}
Observe that the $\tau$-transitions are in correspondence with the transitions
of the evaluation relation~\eqref{eq:red-rel}, and the labeled transitions 
correspond to the remaining transitions of the execution 
relation~\eqref{eq:ex-rel}. 

We now recall the definition of \emph{weak bisimulation relation} 
from~\cite[Section~4.2]{milner1990operational}.
\begin{definition}$ $
\begin{itemize}
\item For processes $p,q$ we write 
$p\ttranss q$ for $p\ttrans^{\raisebox{-3pt}{$\scriptstyle*$}} q$, and for $\alpha\neq\tau$ we write 
$p\xRightarrow{\alpha}q$ for $\exists p',q'\qdot p\ttranss  p'\xrightarrow{\alpha}q'\ttranss q$.
\item
A \emph{weak bisimulation} on $\kkproce$ is a binary relation
$R\subs\kkproce^2$ such that for all $\alpha\in\actions$ and $(p,q)\in R$ 
we have 
\begin{equation}\label{eq:wbs}
\begin{aligned}
p\atrans p' &\qimp\exists q'\qdot q\atranss q'\,\wedge\, 
(p',q')\in R\quad\text{and}\\
q\atrans q' &\qimp\exists p'\qdot p\atranss p'\,\wedge\, (p',q')\in R.
\end{aligned}
\end{equation}
\item Two processes $p,q$ are called \emph{weakly bisimilar} 
(written $p\approx q$), if there exists a weak bisimulation relation $R$ with 
$(p,q)\in R$.
\end{itemize}
\end{definition}
We recall the following important properties of the weak bisimilarity
relation $\approx$.
\begin{lemma}\label{lem:wbs-erel}
Weak bisimilarity is itself a weak bisimulation, and furthermore it is an 
equivalence relation.
\end{lemma}
\begin{proof}
\cite[Proposition~4.2.7]{milner1990operational}
\end{proof}

To show that $\beta$-equivalent processes are bisimilar, we have to find a 
bisimulation relation containing $\beta$-equivalence. The following relation
does the job.
\begin{definition}[$\gamma$-equivalence]
$\gamma$-equivalence (written $p\gequiv q$) is the equivalence relation on 
processes that is generated by $\beta$-reduction and $\tau$-transitions.
\end{definition}
\begin{lemma}\label{lem:gamma-bisim}
$\gamma$-equivalence of processes is a weak bisimulation.
\end{lemma}
\begin{proof}
It is sufficient to verify conditions~\eqref{eq:wbs} on the generators of
$\gamma$-equivalence, i.e.\ one-step $\beta$-reductions and $\tau$-transitions.
Therefore we show the following:
\begin{enumerate}
 \item if $p\ttrans q$ and $p\atrans p'$ then there exists $q'$ with 
 $q\atranss q'$ and $p'\gequiv q'$
 \item if $p\ttrans q$ and $q\atrans q'$ then there exists $p'$ with 
 $p\atranss p'$ and $p'\gequiv q'$
 \item if $p\to_\beta q$ and $p\atrans p'$ then there exists $q'$ with 
 $q\atranss q'$ and $p'\gequiv q'$
 \item if $p\to_\beta q$ and $q\atrans q'$ then there exists $p'$ with 
 $p\atranss p'$ and $p'\gequiv q'$
\end{enumerate}

In the first case, the fact that the LTS can only branch if $\readb$ is 
in head position, and this does not involve $\tau$-transitions, implies that 
$\alpha=\tau$ and $p'=q$, and we can choose $q'=q$ as well.
In the second case we have $p\atranss q'$ and thus can choose $p'=q'$.

\medskip

For cases 3 and 4, which we treat simultaneously, we have to analyze the 
structure of $p$ and $q$, which are of the form
$r[(\lambda x\qdot s)t]$ and $r[s[t/x]]$ for some context $r\chole$ (see Section~\ref{sec:beta}). 
The proof proceeds by cases on the structure of $r\chole$.

If $r\chole$ is of either of the forms $(st\star\pi)\chole$\footnote{The notation is 
meant to convey that we don't care if the hole is in $s$, $t$, or $\pi$.}, 
$\writez\star\pi\chole$, $\writeo\star\pi\chole$, $\comccc\star\pi\chole$, 
$(\comk_\pi\star\rho)\chole$, or $\tend\star\pi\chole$,
then it is immediately evident that $p$ and $q$ can perform the same unique 
transition (if any), and the results will again be $\beta$-equivalent (possibly
trivially, since the redex can get deleted in the transition).

If $r\chole$ is of the form $((\lambda y\qdot u)\star \pi)\chole$ then this is true
as well, regardless of whether the hole is in $u$ or in $\pi$ (here the redex
can be duplicated, if the hole is in the first term in $\pi$).

If $r\chole$ is of the form $\readb\star\pi\chole$ then several transitions may be possible,
but any transition taken by either of $p$ or $q$ can also be taken by the other,
and the results will again be $\beta$-equivalent.

It remains to consider $r\chole$ of the form $\chole\star\pi$. In this case, 
$p=(\lambda x\qdot s)t\star\pi$ and $q=s[t/x]\star\pi$. Here $p$ can perform the
transition $p\ttrans(\lambda x\qdot s)\star t\ap\pi$ which can be matched
by $q\ttranss q$ where we have 
$(\lambda x\qdot s)\star t\ap\pi\gequiv p\gequiv q$. In the other direction we 
have $p\atranss q'$ for every $q\atrans q'$ since $p\ttranss q$.
\end{proof}
The following definition and corollary makes the link between the 
execution relation~\eqref{eq:ex-rel} and the LTS~\eqref{eq:transsys}.
\begin{definition} 
Two execution contexts $(p,\iota,\omicron)$, $(q,\iota',\omicron')$ are called
\emph{$\top$-equivalent} (written $(p,\iota,\omicron)\tequiv(q,\iota',
\omicron')$), if for all $\iota'',\omicron''\in\zos$ we have
\[
(p,\iota,\omicron)\wiggles(\top,\iota'',\omicron'')
\qtext{iff}(q,\iota',\omicron')\wiggles(\top,\iota'',\omicron'').
\]
\end{definition}
\begin{corollary}\label{cor:tequiv}$ $
\begin{enumerate}
\item $p\approx q$ implies $(p,\iota,\omicron)\tequiv(q,\iota,\omicron)$ for all 
$\iota,\omicron\in\zos$. 
\item $(p,\iota,\omicron)\wiggles (q,\iota',\omicron')$ implies 
$(p,\iota,\omicron)\tequiv(q,\iota',\omicron')$.
\item $(p,\iota,\omicron)\tequiv(\top,\iota',\omicron')$ implies
$(p,\iota,\omicron)\wiggles(\top,\iota',\omicron')$.
\end{enumerate}
\end{corollary}
\begin{proof}
For the first claim we show that
\[
p\approx q,\quad (p,\iota,\omicron)\wiggles(\top,\iota',\omicron')\qtext{implies}
(q,\iota,\omicron)\wiggles(\top,\iota',\omicron')
\]
by induction on the length of 
$(p,\iota,\omicron)\wiggles(\top,\iota',\omicron')$. The base case is clear. For 
the induction step assume that $(p,\iota,\omicron)\wiggle(p^*,\iota^*,
\omicron^*)\wiggles(\top,\iota',\omicron')$. If the initial transition is a
$(\tau)$ in the execution relation~\eqref{eq:ex-rel}, then have 
$p^*\cong q$, $\iota^*=\iota$ and $\omicron^*=\omicron$, and we can apply the induction hypothesis.
If the initial transition corresponds to another clause in~\eqref{eq:ex-rel},
then there is a corresponding transition $p\atrans q$ with $\alpha\in\labels$
in the LTS~\eqref{eq:transsys}, and by bisimilarity there exists a $q^*$
with $q\atranss q^*$ and $p^*\approx q^*$. Now the induction hypothesis
implies $(q^*,\iota^*,\omicron^*)\wiggles(\top,\iota',\omicron')$, and from
$q\atranss q^*$ we can deduce 
$(q,\iota,\omicron)\wiggles(q^*,\iota^*,\omicron^*)$ by cases on $\alpha$.

The second claim follows since $\wiggle$ is deterministic.

The third claim follows since $(\top,\iota',\omicron')$ can not perform any
more transitions.
\end{proof}

\section{Expressivity}\label{sec:exp}

In this section we show that the machine with I/O is Turing complete, i.e.\ that 
every computable $f:\N\pto\N$ can be implemented in the sense of 
Def.~\ref{def:compute} by a process $p$.

Roughly speaking, given $f$, we define a process $p$ that reads the input, transforms 
it into a Church numeral, applies a term $t$ that computes $f$ on the level of
Church numerals, and then writes the result out.

To decompose the task we define terms $R$ and $W$ for reading and writing, with
the properties that 
$(R\ast\pi,\bin(n),\omicron)\tequiv(\cn{n}\star\pi,\ve,\omicron)$ ($\cn{n}$ is 
the $n$-th Church numeral), and $(W\cn{n}\ast\pi,\iota,\ve)\tequiv(\top,\iota,
\bin(n))$ for all $n\in\N$.

Now the naive first attempt to combine $R$ and $W$ with the term $t$ computing
the function would be something like $W(tR)$, but this would only work if the
operational semantics was call by value. The solution is to use Krivine's
\emph{storage operators}~\cite{krivine1991lambda} which where devised precisely
to simulate call by value in call by name, and we use a variation of them.

The following definition introduces the terms $R$ and $W$, after giving
some auxiliary definitions.
\begin{definition}$ $
% \begin{itemize}
% \item 
$E,Z,B,C,H,Y,S$ are $\lambda$-terms satisfying
\begin{equation}
\begin{array}
{l@{\,\,\bequiv\,\,}l@{\qquad}l@{\,\,\bequiv\,\,}l@{\qquad}l@{\,\,\bequiv\,\,}l}
B\,\cn{n}&\cn{2n}                 & E\,\cn{(2n)}\,s\,t & s   & Y\,t&t\,(Y\,t)\\
C\,\cn{n}&\cn{2n+1}               & E\,\cn{(2n+1)}\,s\,t & t\\
H\,\cn{n}&\cn{\mathrm{floor}(n/2)}& Z\,\cn{(0)}\,s\,t & s \\ 
S\,\cn{n}&\cn{n+1}                & Z\,\cn{(n+1)}\,s\,t & t
\end{array}\qquad
\end{equation}
for all terms $s,t$ and $n\in\N$, where $\cn{n}$ is the Church numeral $\lambda fx\qdot f^nx$.\footnote{Such terms exist by elementary 
$\lambda$-calculus, see e.g.~\cite[Chapters~3,4]{hindley2008lambda}. In 
particular, $Y$ is known as \emph{fixed point operator}.}

The terms $F$, $R$, $W$ are defined as follows:
\begin{equation}
\begin{array}
{l@{\,\,=\,\,}l@{\qquad}l@{\,\,\bequiv\,\,}l@{\qquad}l@{\,\,\bequiv\,\,}l}
F & \lambda hy\qdot h(S\,y)\footnotemark\\
R & Y{Q}\,\cn{0}\qtext{where}{Q}=\lambda xn\qdot \readb(x(B\,n))(x(C\,n))n\\
W & YV\qtext{where} V=\lambda xn\qdot Z\,n\,\tend(E\,n(\writez\,x(H\,n))(\writeo\,x(H\,n)))
\end{array}
\end{equation}
\footnotetext{This is (part of) a \emph{storage operator} for Church
numerals~\cite{krivine1991lambda}.}
% \end{itemize}
\end{definition}
\vspace{-5mm}
The next three lemmas explain the roles of the terms $R, F$, and $W$.
\begin{lemma}
For all $n\in\N$, $\pi\in\kkpi$ and $\omicron\in\zos$ we have 
$(R\star\pi,\bin(n),\omicron)\tequiv(\cn{n}\star\pi,\ve,\omicron)$.
\end{lemma}
\begin{proof}
For all $n\in\N$ we have 
$Y Q\,\cn{n}
\bequiv Q(Y Q)\cn{n}
\bequiv\readb(Y Q\cn{(2n)})(Y Q\cn{(2n+1)})\cn{n}$, and thus
\begin{equation}
\begin{array}{l@{\,\,\tequiv\,\,}l}
(Y  Q\,\cn{n}\star\pi,\ve,\omicron) &(\cn{n}\star\pi,\ve,\omicron) \\
(Y  Q\,\cn{n}\star\pi,0  \iota,\omicron) &(Y  Q\,\cn{(2n)} \star\pi,\iota,\omicron) \\
(Y  Q\,\cn{n}\star\pi,1  \iota,\omicron)&(Y  Q\,\cn{(2n+1)}\star\pi,\iota,\omicron) \\
\end{array}
\end{equation}
The claim follows by induction on the length of $\bin(n)$, since 
$\bin(2n)=\bin(n)0$ for $n>0$, and $\bin(2n+1)=\bin(n)1$.
\end{proof}
\begin{lemma}
For $n\in\N$ and $t$ any closed term, we have 
$\cn{n}\,F\,t\,\cn{0}\bequiv t\,\cn{n}$.
\end{lemma}
\begin{proof}
This is because
$
\cn{n}\,F\,t\cn{0}\bequiv
F^n\,t\,\cn{0}\bequiv
t\,(S^n\cn{0})\bequiv
t\,\cn{n}$, where the second step can be shown by induction on $n$.
\end{proof}
\begin{lemma}
For all $n\in\N$, $\pi\in\kkpi$ and $\iota\in\zos$ we have 
$(W\cn{n}\star\pi,\iota,\ve)\tequiv (\top,\iota,\bin{(n)})$.
\end{lemma}
\begin{proof}
We have $W\,\cn{n}\bequiv V\,W\,\cn{n}\bequiv Z\,\cn{n}\,\tend(E\,\cn{n}
(\writez\,W(H\,\cn{n}))(\writeo\,W(H\,\cn{n})))$,
and therefore
\begin{equation}
\begin{array}{l@{\,\,\tequiv\,\,}l@{\,\,\tequiv\,\,}l}
( W\,\cn{0}\star\pi,\iota,\omicron)&(\tend\star\pi,\iota,\omicron)  
&(\top,\iota,\omicron)\\
( W\,\cn{(2n)}\star\pi,\iota,\omicron)&(\writez\, W(H\,\cn{(2n)})\star\pi,\iota,\omicron)&(W(\cn{n})\star\pi,\iota,0\omicron)\quad \text{for }(n>0)\\
( W\,\cn{(2n+1)}\star\pi,\iota,\omicron)&(\writeo\, W(H\,\cn{(2n+1)})\star\pi,\iota,\omicron)&(W(\cn{n})\star\pi,\iota,1\omicron).
\end{array}
\end{equation}
The claim follows again by induction on the length of $\bin(n)$.
\end{proof}
\begin{theorem}\label{theo:turing}
Every computable function $f:\N\pto\N$ can be implemented by a process 
$p$. 
\end{theorem}
\begin{proof}
From~\cite[Thm.~4.23]{hindley2008lambda} we know that there exists a term $t$ 
with $t\,\cn{n}\bequiv\cn{f(n)}$
for $n\in\dom(f)$. The process $p$ is given by $R\star F\ap t\ap \cn{0}\ap F\ap 
W\ap \cn{0}$.
Indeed, for $n\in\dom(f)$ we have
\begin{equation}
\begin{array}{l@{\,\,\tequiv\,\,}l@{\,\,\tequiv\,\,}l}
 (R\star F\ap t\ap \cn{0}\ap F\ap W\ap \cn{0},\bin(n),\ve)
&(\cn{n}\star F\ap t\ap \cn{0}\ap F\ap W\ap \cn{0},\ve,\ve)
&(\cn{n}\, F t \cn{0}\star F\ap W\ap \cn{0},\ve,\ve)
\\&(t\,\cn{n}\star F\ap W\ap \cn{0},\ve,\ve)
&(\cn{f(n)}\star F\ap W\ap \cn{0},\ve,\ve)
\\&(W\,\cn{f(n)}\star\ve,\ve,\ve)&(\top,\ve,\bin(f(n)))
\end{array}
\end{equation}
and we deduce $(R\star F\ap t\ap \cn{0}\ap F\ap W\ap \cn{0},\bin(n),\ve)\wiggles (\top,\ve,\bin(f(n)))$ by Corollary~\ref{cor:tequiv}-3.
\end{proof}

\section{Realizability and triposes}\label{sec:real}

The aim of this section is to describe how the presence of I/O instructions 
allows to define new realizability models, which we do in the categorical 
language of triposes and toposes~\cite{vanoosten2008realizability}.

In Subsection~\ref{sec:krivinereal} we give a categorical reading of Krivine's 
realizability interpretation as described 
in~\cite{krivine2009realizabilitypanorama} and show how it gives rise to 
triposes. In Subsection~\ref{sec:extended-real} we show how the definitions
can be adapted to the syntax and machine with I/O, and how this allows us
to define new realizability models from specifications.

The interpretation of Krivine realizability in terms of triposes is due to 
Streicher~\cite{streicher2013krivine}, and has further been 
explored in~\cite{ferrer2014ordered}. However, the presentation here is more 
straightforward since the constructions and proofs do not rely on
\emph{ordered combinatory algebras}, but directly rephrase Krivine's 
constructions categorically.

\subsection{Krivine's classical realizability}\label{sec:krivinereal}

\emph{Throughout this subsection we work with the 
syntax~\eqref{eq:grammar-basic} without I/O instructions but with stack 
constants.}

Krivine's realizability interpretation is always given relative to a set
of processes called a `pole' -- the choice of pole determines the 
interpretation.
\begin{definition}\label{def:pole}
A \emph{pole} is a set $\kkpole\subs\kkproc$ of processes which is 
\emph{saturated}, in the sense that $p\in\kkpole$ and $p'\kred p$ implies 
$p'\in\kkpole$.
\end{definition}
As Miquel~\cite{miquel2011existential} demonstrated, the pole can be seen as 
playing the role of the parameter $R$ in Friedman's negative 
translation~\cite{friedman1977classically}. 
In the following we assume that a pole $\kkpole$ is fixed.

A \emph{truth value} is by definition a set $S\subs\kkpi$ of stacks. 
Given a truth value $S$ and a term $t$, we write $t\real S$ -- and say `$t$
realizes $S$' -- if $\forall \pi\in S\qdot t\star\pi\in\kkpole$.
We write $S\porth=\setof{t\in\kklam}{t\real S}$ for the set of realizers of 
$\kkpi$. So unlike in Kleene realizability the elements of a truth value are
not its realizers -- they should rather be seen as `refutations', and indeed
larger subsets of $\kkpi$ represent `falser' truth values\footnote{For this 
reason, Miquel~\cite{miquel2011existential,miquel2011forcing} calls the elements 
of $P(\kkpi)$ \emph{falsity values}.}; in particular falsity
is defined as
\begin{equation}\label{eq:falsity}
\qquad\bot\quad=\quad\kkpi.
\end{equation}
Given truth values
$S,T\subs\kkpi$, we define the implication $S\imp T$ as follows.
\begin{equation}\label{eq:implication}
S\imp T \quad=\quad S\porth \ap T = \setof{s\ap \pi}{s\real S,\pi\in T}
\end{equation}
With these definitions we can formulate the following lemma, which relates 
refutations of a truth value $S$ with realizers of its negation.
\begin{lemma}\label{lem:kpi-real}
Given $\pi\in S\subs \kkpi$, we have $\comk_\pi\real S\imp\bot$.
\end{lemma}
\begin{proof}
We have to show that $\comk_\pi\star t\ap\rho\in\kkpole$ for all $t\real S$ and 
$\rho\in\kkpi$. This is because $\comk_\pi\star t\ap\rho\kred t\star\pi$, where
$\pi\in S$ and $t\real S$.
\end{proof}

A \emph{(semantic) predicate} on a set $I$ is a function $\varphi:I\toppi$
from $I$ to truth values. On semantic predicates we define the basic logical 
operations of falsity, implication, universal quantification, and reindexing by
\begin{equation}\label{eq:connectives}
\begin{array}{r@{\quad=\quad}ll}
\bot(i) & \kkpi&\text{\it (falsity)}\\[1mm]
 (\varphi\imp\psi)(i) & \varphi(i)\imp\psi(i)=\varphi(i)\porth\ap\psi(i)&\text{\it (implication)}\\[1mm]
\forall_f(\theta)(i) &\bigcup_{f(j)=i} \theta(j)& \text{\it (universal quantification)}\\
f^*\varphi & \varphi\circ f&\text{\it (reindexing)}\\
\end{array}
\end{equation}
for $\varphi,\psi:I\toppi$, $\theta:J\toppi$ and $f:J\to I$. Thus, for any
function $f:J\to I$, the function $\forall_f$ (called `universal quantification 
along $f$') maps predicates
on $J$ to predicates on $I$\footnote{The usual $\forall x\vtp A$ from predicate 
logic corresponds to taking $f$ to be a projection map $\pi_1:\Gamma\times A\to 
\Gamma$, see e.g.\ \cite[Chapter~4]{jacobs2001categorical}.}, and the function 
$f^*$ (called `reindexing along $f$') maps predicates on $I$ to predicate on 
$J$. We write $\forall_I$ for universal quantification along the terminal 
projection $I\to 1$.

Next, we come to the concept of `truth/validity' of the interpretation. We can
not simply call a truth value `true' if it has a realizer -- this would lead
to inconsistency as soon as the pole $\kkpole$ is nonempty, since 
$\comk_\pi t\real\bot$ for any process $t\star\pi\in\kkpole$. The solution
is to single out a set $\kkpl$ of `well-behaved' realizers called
`proof-like terms'. We recall the definition 
from~\cite{krivine2009realizabilitypanorama}.
\begin{definition}\label{def:prooflike}
The set $\kkpl\subs\kklam$ of \emph{proof-like terms} is the set of terms $t$
that do not contain any continuations $\comk_\pi$.
\end{definition}
As Krivine~\cite[pg.~2]{krivine2009realizabilitypanorama} points out, $t$ is a
proof-like term if and only if it does not contain any stack constant
$\pi_0\in\kkpiz$ (since continuation terms $\comk_\pi$ necessarily contain a 
stack constant at the end of $\pi$, and conversely stacks can only occur as
continuations in a term).

Proof-like terms give us a concept of logical validity -- a truth value $S$ is 
called \emph{valid}, if there exists a proof-like term $t$ with $t\real S$.

With this notion, we are ready to define the centerpiece of the realizability 
model, which is the \emph{entailment relation} on predicates.
\begin{definition}$ $
For any set $I$ and integer $n$, the $(n+1)$-ary entailment
relation $(\ent_I)$ on predicates on $I$ is defined by 
\[\varphi_1\dots\varphi_n\ent_I\psi\qtext{if and only if} 
\exists t\in\kkpl\qdot t\real\forall_I(\varphi_1\imp\dots\imp\varphi_n\imp \psi).\]
If the right hand side proposition holds, we call $t$ a \emph{realizer} of 
$\varphi_1\dots\varphi_n\ent_I\psi$.
\end{definition}
Thus, $\varphi_1\dots\varphi_n\ent_I\psi$ means that the truth value \
$\forall_I(\varphi_1\imp\dots\imp\varphi_n\imp \psi)$ is valid. 
More explicitly this can be written out as 
\begin{equation}\label{eq:entailment-explicit}
\exists t\in\kkpl\fa i\in I,u_1\in\varphi_1(i)\porth,\dots,u_n\in\varphi_n(i)\porth,
\pi\in\psi(i)\qdot
t\star u_1\ap\dots\ap u_n\ap\pi\in\kkpole.
\end{equation}
With the aim to show that the semantic predicates form a tripos in 
Theorem~\ref{theo:tripos}, we now prove that the entailment ordering models the 
logical rules in Table~\eqref{tab:erelrules}. The first eight rules form a
standard natural deduction system for (the $\bot,\imp$ fragment of) classical 
propositional logic, but for universal quantification we give categorically
inspired rules that bring us quicker to where we want, and in particular
avoid having to deal with variables.
\newcommand{\rname}[1]{\RightLabel{\textrm{(#1)}}}
\begin{table}[h]
\noindent\framebox[\textwidth]{
\parbox{.96\textwidth}{
\begin{center}
\begin{tabular}{@{}cc@{}}
\hspace{-.34cm}
\AXC{$\phantom{|}$}
\rname{Ax}
\UIC{$\varphi\enti\varphi$}
\DisplayProof
&
\AXC{$\Gamma\enti\bot$}
\rname{$\bot$E}
\UIC{$\Gamma\enti\psi$}
\DisplayProof
\\[\bigskipamount]
\AXC{$\Gamma,\varphi\enti\psi$}
\rname{$\imp$I}
% \doubleLine
\UIC{$\Gamma\enti\varphi\imp\psi$}
\DP
&
\AXC{$\Gamma\enti\psi$}
\AXC{$\Delta\enti\psi\imp\theta$}
\rname{$\imp$E}
\BIC{$\Gamma,\Delta\enti\theta$}
\DP 
\\[\bigskipamount]
\AXC{$\Gamma\enti \psi$}
\rname{S}
\UIC{$\sigma(\Gamma)\enti \psi$}
\DP 
&
\AXC{$\phantom{|}$}
\rname{PeL}
\UIC{$\Delta\csep\Gamma\enti ((\psi\imp\bot)\imp\psi)\imp\psi$}
\DP\\[\bigskipamount]
\AXC{$\Gamma\enti \psi$}
\rname{W}
\UIC{$A,\Gamma\enti \psi$}
\DP &
\AXC{$A,A,\Gamma\enti \psi$}
\rname{C}
\UIC{$A,\Gamma\enti \psi$}
\DP \\[\bigskipamount]
\AXC{$f^*\Gamma\entj\xi$}
\rname{$\forall$I}
\UIC{$\Gamma\enti\forall_f\xi$}
\DP
&
\AXC{$\Gamma\enti\forall_f\xi$}
\rname{$\forall$E}
\UIC{$f^*\Gamma\entj\xi$}
\DP
\end{tabular}
\end{center}
$\varphi,\psi,\theta$ are predicates on $I$, i.e.\ functions
$I\toppi$, and $\Gamma\equiv\varphi_1\dots\varphi_n$ and $\Delta\equiv\psi_1\dots\psi_m$
are lists of such predicates. $\xi$ is a predicate on $J$, and $f:J\to I$ is a function. $\sigma$
is a permutation of $\{1,\dots,n\}$. $f^*\Gamma$ is an abbreviation for
$f^*\varphi_1\dots f^*\varphi_n$, and $\sigma(\Gamma)$ is an abbreviation for
$\varphi_{\sigma(1)}\dots\varphi_{\sigma(n)}$.
\medskip}}
\caption{Admissible rules for the entailment relation.}\label{tab:erelrules}
\end{table}
% \vspace{-5mm}
\begin{lemma}\label{lem:rules-valid}
The rules displayed in Table~\ref{tab:erelrules} are admissible for the 
entailment relation, in the sense that if the hypotheses hold then so 
does the conclusion.
\end{lemma}
\begin{proof}
 (Ax) rule: The conclusion is realized by $\lambda x\qdot x$.

($\bot$E) rule: every realizer of the hypothesis is also a realizer of the
conclusion, since $\psi(i)\subs\bot(i)=\kkpi$ for all $i\in I$.

($\imp$I) rule: the hypothesis and the conclusion have precisely the same 
realizers.

($\imp$E) rule: if $t$ realizes $\Delta\enti\psi\imp\theta$ and $u$ realizes 
$\Gamma\ent_i\psi$ then $\Gamma,\Delta\ent_i\theta$ is realized by
$\lambda x_1\dots x_ny_1\dots y_m\qdot ty_1\dots y_m(u x_1\dots x_n)$.

(PeL) rule (`Peirce's law'): the conclusion is realized by $\comccc$. 
To see this, let $i\in I$, $t\real(\psi(i)\imp\bot)\imp\psi(i)$, and 
$\pi\in\psi(i)$.
Then we have $\comccc\star t\ap\pi\kred t\star\comk_\pi\ap\pi$, which is in $\kkpole$
since $\comk_\pi\ap\pi\in(\psi(i)\imp\bot)\imp\psi(i)$ by 
Lemma~\ref{lem:kpi-real} and the definition~\eqref{eq:connectives} of 
implication.

(W) rule: if $t$ realizes $\Gamma\enti\psi$, then $\lambda x\qdot t$ realizes
$A,\Gamma\enti\psi$.

(C) rule: if $t$ realizes $A,A,\Gamma\enti\psi$, then $\lambda x\qdot txx$
realizes $A,\Gamma\enti\psi$.

(S) rule: if $t$ realizes $\Gamma\enti\psi$, then 
$\lambda x_{\sigma(1)}\dots x_{\sigma(n)}\qdot tx_1\dots x_n$ realizes 
$\sigma(\Gamma)\enti\psi$.

($\forall$I) and ($\forall$E) rules: $\Gamma\enti\forall_f\xi$ and 
$f^*\Gamma\entj\xi$ have exactly the same realizers. 
Indeed, a realizer of $f^*\Gamma\ent_J \xi$ is a term $t$ satisfying
\[
\fa j\in J,u_1\in\varphi_1(f(j))\porth,\dots,u_n\in\varphi_n(f(j))\porth,
\pi\in\xi(j)\qdot
t\star u_1\ap\dots\ap u_n\ap\pi\in\kkpole,
\]
and a realizer of $\Gamma\entj\forall_f\xi$ is a term $t$ satisfying
\[
\textstyle\fa i\in I,u_1\in\varphi_1(i)\porth,\dots,u_n\in\varphi_n(i)\porth,
\pi\in\bigcup_{f(j)=i}\xi(j)\qdot
t\star u_1\ap\dots\ap u_n\ap\pi\in\kkpole,
\]
and both statements can be rephrased as a quantification over pairs $(i,j)$ with
$f(j)=i$.
\end{proof}
We only defined the propositional connectives $\bot,\imp$, since 
$\top,\wedge,\vee,\neg$ can be encoded as follows:
\begin{equation}\label{eq:encodings}
\begin{array}{r@{\defcon}l@{\qquad}r@{\defcon}l@{\qquad}r@{\defcon}l}
\top &   \bot\imp\bot & \neg\varphi &  \varphi\imp\bot 
\\
\varphi\wedge\psi &  (\varphi\imp(\psi\imp\bot))\imp\bot&\varphi\vee\psi &  
(\varphi\imp\bot)\imp\psi \\
\end{array}
\end{equation}
With these encodings it is routine to show the following.
\begin{lemma}\label{lem:classical-nd}
With the connectives $\top,\wedge,\vee,\neg$ encoded as in~\eqref{eq:encodings}, 
the rules of propositional classical natural deduction (e.g.\ system Nc 
in~\cite[Section~2.1.8]{troelstra2000basic}) are derivable from the rules in
Table~\ref{tab:erelrules}.
\end{lemma}
% This result does only apply to the propositional connectives and not to 
% $\exists$, since already the rules for $\forall$ in Table~\ref{tab:erelrules}
% are not the usual ones, whose treatment would require a careful inductive
% interpretation 

With this we can show that for any set $I$, the binary part of the
entailment relation makes $P(\kkpi)^I$ into a \emph{Boolean prealgebra}.
\begin{definition} \label{def:boolean-algebra}
A \emph{Boolean prealgebra} is a 
preorder
$(B,\leq)$ which
\begin{enumerate}
\item has binary \emph{joins} and \emph{meets} -- denoted by $x\vee y$ and 
$x\wedge y$ for $x,y\in D$,
\item has a \emph{least element} $\bot$ and a \emph{greatest element} $\top$,
\item is \emph{distributive} in the sense that 
$
x\wedge (y\vee z)\cong (x\wedge y)\vee(x\wedge z)
$
{for all} $x,y,z\in B$, and 
\item is \emph{complemented}, i.e.\ for every $x\in D$ there exists
a $\neg x$ with $x\wedge \neg x\cong\bot$ and $x\vee\neg x\cong\top$.
\end{enumerate}
\end{definition}
\begin{lemma}\label{lem:boolean}
Writing $\varphi\leq\psi$ for $\varphi\enti\psi$,
$(P(\kkpi)^I,\leq)$ is a Boolean prealgebra.
\end{lemma}
\begin{proof}
The (Ax) rule implies that $\leq$ is reflexive, and transitivity follows
from the derivation
\vspace{-1mm}
\[
\AXC{$\varphi\enti\psi$}
\AXC{$\psi\enti\theta$}
\UIC{$\enti\psi\imp\theta$}
\BIC{$\varphi\enti\theta$}
\DP.\quad\text{}
\]
\vspace{-1mm}
Thus, $\leq$ is a preorder on $P(\kkpi)^I$.

The joins, meets, complements, and least and greatest element are given by
the corresponding logical operations as defined in~\eqref{eq:connectives} 
and~\eqref{eq:encodings}.

The required properties all follow from derivability of corresponding entailments
and rules in classical natural deduction -- for example, $\varphi\wedge\psi$ is
a binary meet of $\varphi$ and $\psi$ since 
\begin{equation}
(\ast)\text{ the entailments }\varphi\wedge\psi\enti\varphi\text{ and }
\varphi\wedge\psi\enti\psi\text{ and the rule }\,
\AXC{$\theta\enti\varphi$}
\AXC{$\theta\enti\psi$}
\BIC{$\theta\enti\varphi\wedge\psi$}
\DP
\end{equation}\vspace{-1mm}
are derivable.

Distributivity follows from derivability of the entailments
$
\varphi\wedge(\psi\vee\theta)\enti(\varphi\wedge\psi)\vee(\varphi\wedge\theta)
$ and $
(\varphi\wedge\psi)\vee(\varphi\wedge\theta)\enti\varphi\wedge(\psi\vee\theta).
$
\end{proof}
We now come to \emph{triposes}, which are a kind categorical model
for higher order logic. We use a `strictified' version of the original
definition~\cite[Def.~1.2]{hjp80} since this bypasses some subtleties and is
sufficient for our purposes. Furthermore, we are only interested modeling
classical logic here, and thus can restrict attention to triposes whose fibers
are Boolean (instead of Heyting) prealgebras.
\begin{definition}
A \emph{strict\footnote{`Strict' refers to the facts that (i) 
$\trip$ is a \emph{functor}, not merely a pseudofunctor (ii) the Boolean prealgebra 
structure is preserved `on the nose' by the monotone maps
$\trip(f)$ (iii) the Beck-Chevalley condition is required up to equality, not
merely isomorphism, (iv) we require equality and uniqueness in the last 
condition. Every strict tripos is a tripos in the usual sense, and conversely
it can be shown that any tripos is equivalent to a strict one.} 
Boolean
tripos} is a contravariant functor
$
\trip:\setord
$
from the category of sets to the category of preorders such that
\begin{itemize}
\item for every set $I$, the preorder $\trip(I)$ is a Boolean prealgebra, and
for any function $f:J\to I$, the induced monotone map 
$\trip(f):\trip(I)\to\trip(J)$ preserves all Boolean prealgebra structure.
\item for any $f:J\to I$, $\trip(f)$ has left and right 
adjoints\footnote{`Adjoint' in the sense of `adjoint functor', where monotone
maps are viewed as functors between degenerate
categories.}
$\exists_f\adj\trip(f)\adj\forall_f$ such that 
\begin{equation}\label{eq:bc-cond}
\text{for any pullback square\footnotemark}\quad
\vcenter{\xymatrix@R-3mm{
L\ar[r]_q\ar[d]_p & K\ar[d]^g \\
J\ar[r]^f & I
}}
% \begin{tikzcd}
% L \arrow[r,"q"'] \arrow[d,"p"'] 
% & K \arrow[d,"g"] \\
% J \arrow[r,"f"]
% & I
% \end{tikzcd}
\end{equation}
\footnotetext{%
The square being a pullback means that $f\circ p=g\circ q$ and 
$\forall jk\qdot f(j)=g(k)\imp\exists! l\qdot p(l)=j\wedge q(l)=k$.
}
we have $\trip(g)\circ\forall_f=\forall_q\circ\trip(p)$ (this is the
\emph{Beck-Chevalley condition}), and
\item there exists a \emph{generic predicate}, i.e.\ a set $\prop$ and
an element $\triptr\in\trip(\prop)$ such that for every set $I$ and 
$\varphi\in\trip(I)$ there exists a unique function $f:I\to\prop$ with 
$\trip(f)(\triptr)=\varphi$.
\end{itemize}
\end{definition}
\medskip
The assignment $I\mapsto(P(\kkpi)^I,\leq)$ extends to a functor
$\trip_\kkpole:\setord$ by letting $\trip_\kkpole(f)=f^*$, i.e.\ mapping every 
function $f:J\to I$ to the reindexing function along $f$, which is monotone 
since every realizer of $\varphi\enti\psi$ is also a
realizer of $\varphi\circ f\entj\psi\circ f$. 
\begin{theorem}\label{theo:tripos}
$\trip_\kkpole$ is a strict Boolean tripos.
\end{theorem}
\begin{proof}
We have shown in Lemma~\ref{lem:boolean} that the preorders $(P(\kkpi)^I,\leq)$
are Boolean prealgebras. It is immediate from~\eqref{eq:connectives} that 
the reindexing functions $f^*$ preserve $\bot$ and $\imp$, and the other Boolean 
operations are preserved since they are given by encodings.

The identity function $\id:P(\kkpi)\toppi$ is a generic predicate for
$\trip_\kkpole$.

The ($\forall$I) and ($\forall$E) rules together imply
that the operation $\forall_f:P(\kkpi)^I\toppi^J$ is right adjoint
to $f^*$ for any $f:J\to I$. Existential quantification along $f:J\to I$
is given by $\exists_f =\neg\circ\forall_f\circ\neg$, which is left adjoint to
$f^*$ since
\[
\neg\forall_f\neg\varphi\enti\psi\qtext{iff}
\neg\psi\enti\forall_f\neg\varphi\qtext{iff}
f^*\neg\psi\entj\neg\varphi \qtext{iff}
\neg f^*\psi\entj\neg\varphi \qtext{iff}
\varphi\entj f^*\psi
\]
for all $\varphi:J\toppi$ and $\psi:I\toppi$.

It remains to verify the Beck-Chevalley condition. Given a square as 
in~\eqref{eq:bc-cond} we have
\[
\textstyle g^*\forall_f(\varphi(k))=\bigcup\setof{\varphi(j)}{f(j)=g(k)}
\qtext{and}
\forall_q(p^*(k))=\bigcup\setof{\varphi(j)}{\exists l\qdot pl=j\wedge ql= k},
\]
and the two terms are equal since the square is a pullback.
\end{proof}
Thus we obtain a tripos $\trip_\kkpole$ \emph{for each pole} $\kkpole$.
As Hyland, Johnstone, and Pitts showed in~\cite{hjp80}, every tripos $\trip$
gives rise to a topos $\catset[\trip]$ via the \emph{tripos-to-topos
construction}. Since the fibers of the triposes $\trip_\kkpole$ are Boolean 
prealgebras, the toposes $\catset[\trip_\kkpole]$ are Boolean as well, which means
that their {internal logic} is classical.

\subsubsection{Consistency}

Triposes of the form $\trip_\kkpole$ can be degenerate in two ways: if 
$\kkpole$ is empty then $\trip_\kkpole(I)\simeq(P(I),\subs)$ for every set $I$, 
and the topos $\catset[\trip_\kkpole]$ is equivalent to the category $\catset$.

If, in the other extreme, the pole is so big that there exists a proof-like 
$t$ realizing $\bot$, i.e.\ falsity is valid in the model, then we have 
$\trip_\kkpole(I)\simeq 1$ for all $I$ (since $t$ realizes every entailment
$\varphi\enti\psi$), and the topos $\catset[\trip_\kkpole]$ 
is equivalent to the terminal category.

By \emph{consistency} we mean that falsity is \emph{not} valid, or equivalently 
that
\begin{equation}\label{eq:consistency}
\forall t\in\kkpl\ex\pi\in\kkpi\qdot t\star\pi\not\in\kkpole.
\end{equation}
The `canonical' (according to Krivine~\cite{krivine2011ralgsii}) non-trivial
consistent pole is the \emph{thread model}, which is given by postulating a 
stack constant $\pi_t$ for each proof-like term $t$, and defining
$
\kkpole=\setof{p\in\kkproc}{\neg\exists t\in\kkpl\qdot t\star\pi_t\wiggles p}.
$
Then the processes $t\star\pi_t$ are not in $\kkpole$ for any proof-like $t$, which
ensures the validity of condition~\eqref{eq:consistency}.

In the next section we show how the presence of side effects allows to define
a variety of new, `meaningful' consistent poles.

\subsection{Krivine realizability with I/O}\label{sec:extended-real}

The developments of the previous section generalize pretty much directly to the
syntax with I/O. Concretely, we carry over the definitions of \emph{pole},
\emph{truth value}, \emph{realizer}, \emph{predicate}, and of the basic
logical operations $\bot,\imp,\forall$, by replacing $\kklam$ with $\kklame$,
$\kkpi$ with $\kkpie$, and $\kkproc$ with $\kkproce$.

We point out that in presence of effects, Definition~\ref{def:pole}
only means that $\kkpole$ is saturated w.r.t.\ \emph{effect-free} evaluation,
in contrast to Miquel's approach~\cite{miquel2009modaleffects}
where a pole is a set of (what we call) execution contexts, closed under
the entire execution relation.

The concept of proof-like term deserves some reexamination. It turns out that
the appropriate concept of \emph{proof-like term} is `term not containing
any side effects'. This is consistent with Definition~\ref{def:prooflike} if we 
read `free of side effects' as `free of non-logical constructs', which are the 
stack constants in Krivine's case. Continuation terms, on the other hand, can be 
considered proof-like. We redefine therefore:
\begin{definition}
The set $\kkpl\subs\kklame$ of \emph{proof-like terms} is the set of terms not 
containing any of the constants $\readb,\writez,\writeo,\tend$.
\end{definition}
With this rephrased definition of proof-like term, we can define the entailment
relation on the extended predicates in the same way:
\begin{definition}$ $
For any set $I$ and integer $n$, the $(n+1)$-ary entailment
relation $(\ent_I)$ on the set $P(\kkpie)^I$ of extended predicates on $I$ is 
defined by 
\vspace{-1mm}
\[\varphi_1\dots\varphi_n\ent_I\psi\qtext{if and only if} \exists t\in\kkpl\qdot 
t\real\forall_I(\varphi_1\imp\dots\imp\varphi_n\imp \psi).\]
\vspace{-1mm}
As a special case, the ordering on extended predicates is defined by
\[
\varphi\leq\psi\qtext{if and only if} \exists t\in\kkpl\qdot 
t\real\forall_I(\varphi\imp \psi).
\]
\end{definition}
\vspace{-1mm}
With these definitions, we can state analogues of Lemma~\ref{lem:boolean}
and Theorem~\ref{theo:tripos}:
\begin{theorem}$ $
\begin{itemize}
\item For each set $I$, the order $(P(\kkpie)^I,\leq)$ of extended predicates
is a Boolean prealgebra.
\item 
The assignment $I\mapsto(P(\kkpie)^I,\leq)$ gives rise to a strict Boolean 
tripos $\trip_\kkpole:\setord$.
\end{itemize}
\end{theorem}
\begin{proof}
This follows from the arguments in Section~\ref{sec:krivinereal}, since the 
proofs of Lemmas~\ref{lem:kpi-real},\ref{lem:rules-valid},
\ref{lem:classical-nd},\ref{lem:boolean}, and of Theorem~\ref{theo:tripos} are 
not obstructed in any way by the new constants, nor do they rely on stack 
constants. The redefinition of `proof-like term' does not cause any problems 
either, since we never relied on proof-like terms not containing continuation 
terms.
\end{proof}

The above rephrasing of the definition of proof-like term admits an intuitive 
reformulation of the consistency criterion~\eqref{eq:consistency}:
\begin{lemma}\label{lem:consistency}
A pole $\kkpole$ is consistent if and only if every 
$p\in\kkpole\setminus\{\top\}$ contains a non-logical constant, i.e.\ one of 
$\readb,\writez,\writeo,\tend$.
\end{lemma}
\begin{proof}
If every element of $p\in\kkpole\setminus\{\top\}$ contains a non-logical 
constant, then $t\star\ve$ is not in $\kkpole$ for any proof-like $t$, which 
implies~\eqref{eq:consistency}.

On the other hand, if $t\ast\pi\in\kkpole$ does not contain any non-logical 
constant then $\comk_\pi t$ is a proof-like term which realizes $\bot$, since
for any $\rho\in\kkpie$ we have 
$\comk_\pi t\star\rho\kred\comk_\pi\star t\ap\rho\kred t\ast\pi\in\kkpole$.
\end{proof}

\subsubsection{Poles from specifications}\label{se:polespec}

The connection between poles and specifications is established by the following
lemma.
\begin{lemma}\label{lem:spec-pole}\label{lem:bisim-pole}
Every set $\kkpole$ of processes that is closed under weak bisimilarity is a 
pole.
\end{lemma}
\begin{proof}
This is because $p\approx q$ whenever $p\kred q$, which follows 
from Lemma~\ref{lem:gamma-bisim}.
\end{proof}
Since we can assume that for any reasonable specification the processes 
implementing it are closed under weak bisimilarity, we can thus conclude that
for any specification, the set of processes implementing it is a pole. For 
example:
\begin{itemize}
\item $\kkpole_{\mathrm{cp}}$ is the set of processes that read the input, copy 
every bit immediately to the output, and terminate when the input is empty. We 
have 
$Y\star(\lambda x\qdot \readb(\writez\,x)(\writeo\,x)\tend)\in\kkpole_{\mathrm{cp}}$.
\item $\kkpole_{\mathrm{cp'}}$ contains the processes that first read the
entire input, and {then} write out the same string and terminate. 
We have $R\star F\ap W\ap\cn{0}\in\kkpole_{\mathrm{cp'}}$
with the notations of Section~\ref{sec:exp}.
\item For any partial function $f:\N\pto\N$, the pole $\kkpole_f$ consists of 
those processes that implement $f$ in the sense of Definition~\ref{def:compute}.
\item Since poles are closed under unions, we can define the pole
$\kkpole_F=\bigcup_{f\in F}\kkpole_f$ for any set $F\subs(\N\pto\N)$ of partial 
functions. 
\end{itemize}

\subsubsection{Toposes from computable functions}

We are particularly interested in the poles $\kkpole_f$ associated to 
computable functions $f$, and we want to use the associated triposes 
$\trip_f=\trip_{\kkpole_f}$ and toposes $\catset[\trip_f]$ to study these 
functions.

The following theorem provides a first `sanity check', in showing that the
associated models are non-degenerate.
\begin{theorem}\label{theo:f-com-con} 
Let $f:\N\pto\N$.
\begin{itemize}
\item $\kkpole_f$ is consistent if and only if $f$ is not totally undefined.
\item $\kkpole_f$ is non-empty if and only if $f$ is computable.
\end{itemize}
\end{theorem}
\begin{proof}
The first claim follows from Lemma~\ref{lem:consistency}. If $n\in\dom(f)$ and
$t\star\pi$ implements $f$, then $(t\star\pi,\bin(n),\ve)$ must terminate and 
thus $t\star\pi$ must contain an $\tend$ instruction. The totally undefined 
function, on the other hand, is by definition implemented by every process.

For the second claim, we have shown in Theorem~\ref{theo:turing} that every
computable $f$ is implemented by some process. Conversely, every implementable
function is computable since the Krivine machine with I/O is an effective model
of computation.
\end{proof}

\subsection{Discussion and future work}

The structure and properties of the toposes $\catset[\trip_f]$ remain 
mysterious for the moment, and in future work we want to explore which kind
of properties of $f$ are reflected in $\catset[\trip_f]$. In the spirit of 
Grothendieck~\cite{grothendieck1972theorie} we want to view the 
toposes $\catset[\trip_f]$ as \emph{geometric} rather than logical objects, 
the guiding intuition being that $\catset[\trip_f]$ can be seen as 
representation of 
\emph{`the space of solutions
to the algorithmic problem of computing $f$'},
encoding e.g.\ information on how algorithms computing $f$ can be decomposed
into simpler parts. 

Evident problems to investigate are to understand the lattice of truth values
in $\catset[\trip_f]$, and to determine for which pairs $f,g$ of functions 
the associated toposes are equivalent, and which functions can be separated.

A more audacious goal is to explore whether $\catset[\trip_f]$ can teach us
something about the complexity of a computable function $f$. The Krivine machine 
with I/O seems to be a model of computation that is fine grained enough to 
recognize and differentiate time complexity of different implementations of $f$, 
but it remains to be seen in how far this information is reflected in 
the `geometry' of $\catset[\trip_f]$.

\subsection*{Acknowledgements}

Thanks to Jakob Grue Simonsen and Thomas Streicher for many discussions. 

\bibliographystyle{plain}% the recommended bibstyle
\bibliography{tlca-2015}

\begin{thebibliography}{10}

\bibitem{barendregt1984lambda}
H.P. Barendregt.
\newblock {\em The lambda calculus, Its syntax and semantics}, volume 103 of
  {\em Studies in Logic and the Foundations of Mathematics}.
\newblock North-Holland Publishing Co., Amsterdam, revised edition, 1984.

\bibitem{ferrer2014ordered}
W.~Ferrer, J.~Frey, M.~Guillermo, O.~Malherbe, and A.~Miquel.
\newblock Ordered combinatory algebras and realizability.
\newblock {\em arXiv preprint arXiv:1410.5034}, 2014.

\bibitem{friedman1977classically}
H.~Friedman.
\newblock Classically and intuitionistically provably recursive functions.
\newblock In {\em Higher set theory ({P}roc. {C}onf., {M}ath.
  {F}orschungsinst., {O}berwolfach, 1977)}, volume 669 of {\em Lecture Notes in
  Math.}, pages 21--27. Springer, Berlin, 1978.

\bibitem{grothendieck1972theorie}
A.~Grothendieck, M.~Artin, and J.L. Verdier.
\newblock Th{\'e}orie des topos et cohomologie {\'e}tale des sch{\'e}mas.
\newblock {\em Lecture Notes in Mathematics}, 269, 1972.

\bibitem{hindley2008lambda}
J.~Roger Hindley and Jonathan~P. Seldin.
\newblock {\em Lambda-calculus and combinators, an introduction}.
\newblock Cambridge University Press, Cambridge, 2008.

\bibitem{hyland82}
J.M.E. Hyland.
\newblock The effective topos.
\newblock In {\em The {L}.{E}.{J}. {B}rouwer {C}entenary {S}ymposium
  ({N}oordwijkerhout, 1981)}, volume 110 of {\em Stud. Logic Foundations
  Math.}, pages 165--216. North-Holland, Amsterdam, 1982.

\bibitem{hjp80}
J.M.E. Hyland, P.T. Johnstone, and A.M. Pitts.
\newblock Tripos theory.
\newblock {\em Math. Proc. Cambridge Philos. Soc.}, 88(2):205--231, 1980.

\bibitem{jacobs2001categorical}
B.~Jacobs.
\newblock {\em {Categorical logic and type theory}}.
\newblock Elsevier Science Ltd, 2001.

\bibitem{krivine1991lambda}
J.L. Krivine.
\newblock Lambda-calcul, \'evaluation paresseuse et mise en m\'emoire.
\newblock {\em RAIRO Inform. Th\'eor. Appl.}, 25(1):67--84, 1991.

\bibitem{krivine2009realizabilitypanorama}
J.L. Krivine.
\newblock Realizability in classical logic.
\newblock {\em Panoramas et synth{\`e}ses}, 27:197--229, 2009.

\bibitem{krivine2010ralgs}
J.L. Krivine.
\newblock Realizability algebras: a program to well order {$\mathbb{R}$}.
\newblock {\em Log. Methods Comput. Sci.}, 7(3):3:02, 47, 2011.

\bibitem{krivine2011ralgsii}
J.L. Krivine.
\newblock Realizability algebras {II}: {N}ew models of
  {$\mathrm{ZF}+\mathrm{DC}$}.
\newblock {\em Log. Methods Comput. Sci.}, 8(1):1:10, 28, 2012.

\bibitem{milner1990operational}
R.~Milner.
\newblock Operational and algebraic semantics of concurrent processes.
\newblock In {\em Handbook of theoretical computer science, {V}ol.\ {B}}, pages
  1201--1242. Elsevier, Amsterdam, 1990.

\bibitem{miquel2009modaleffects}
A.~Miquel.
\newblock Classical modal realizability and side effects.
\newblock {\em preprint}, 2009.

\bibitem{miquel2011existential}
A.~Miquel.
\newblock Existential witness extraction in classical realizability and via a
  negative translation.
\newblock {\em Log. Methods Comput. Sci.}, 7(2):2:2, 47, 2011.

\bibitem{miquel2011forcing}
A.~Miquel.
\newblock Forcing as a program transformation.
\newblock In {\em 26th {A}nnual {IEEE} {S}ymposium on {L}ogic in {C}omputer
  {S}cience---{LICS} 2011}, pages 197--206. IEEE Computer Soc., Los Alamitos,
  CA, 2011.

\bibitem{stekelenburg2013realizability}
W.P. Stekelenburg.
\newblock {\em Realizability Categories}.
\newblock PhD thesis, Utrecht University, 2013.

\bibitem{streicher2013krivine}
T.~Streicher.
\newblock Krivine's classical realisability from a categorical perspective.
\newblock {\em Mathematical Structures in Computer Science}, 23(06):1234--1256,
  2013.

\bibitem{troelstra2000basic}
A.S. Troelstra and H.~Schwichtenberg.
\newblock {\em Basic proof theory}, volume~43 of {\em Cambridge Tracts in
  Theoretical Computer Science}.
\newblock Cambridge University Press, Cambridge, 1996.

\bibitem{vanoosten2012classical}
J.~van Oosten.
\newblock \emph{Classical Realizability}.
\newblock Invited talk at ``Cambridge Category Theory Seminar'', slides at
  \url{http://www.staff.science.uu.nl/~ooste110/talks/cambr060312.pdf}.

\bibitem{vanoosten2008realizability}
J.~van Oosten.
\newblock {\em {Realizability: {A}n {I}ntroduction to its {C}ategorical
  {S}ide}}.
\newblock Elsevier Science Ltd, 2008.

\end{thebibliography}
\end{document}